\documentclass[12pt,a4paper]{amsart}
\usepackage{a4wide}
\usepackage{amsmath, amssymb, amsfonts,enumerate}
\usepackage[all]{xy}
\usepackage{amscd}
\usepackage{comment}

\newtheorem{theorem}{Theorem}[section]
\newtheorem{lemma}[theorem]{Lemma}
\newtheorem{corollary}[theorem]{Corollary}

 \theoremstyle{definition}
 
 \newtheorem{remark}[theorem]{Remark}

 \newtheorem{example}[theorem]{Example}

\numberwithin{equation}{section}
\newcommand {\N}{\mathbb{N}} 
\newcommand {\Z}{\mathbb{Z}} 
\newcommand {\Q}{\mathbb{Q}} 

\newcommand{\CC}{\mathcal{C}}

\DeclareMathOperator{\Inv}{Inv}
\DeclareMathOperator{\CA}{CA}

\DeclareMathOperator{\End}{End}
\DeclareMathOperator{\Mor}{Mor}

\DeclareMathOperator{\Sym}{Sym}

\DeclareMathOperator{\Map}{Map}

\begin{document}
\title{On  Residually Finite Semigroups of Cellullar Automata}
\author{Tullio Ceccherini-Silberstein}
\address{Dipartimento di Ingegneria, Universit\`a del Sannio, C.so
Garibaldi 107, 82100 Benevento, Italy}
\email{tceccher@mat.uniroma3.it}
\author{Michel Coornaert}
\address{Institut de Recherche Math\'ematique Avanc\'ee,
UMR 7501,                                             Universit\'e  de Strasbourg et CNRS,
                                                 7 rue Ren\'e-Descartes,
                                               67000 Strasbourg, France}
\email{coornaert@math.unistra.fr}
\subjclass[2010]{20E26, 20M30, 37B10, 37B15, 68Q80}
\keywords{cellular automaton, monoid, semigroup, residual finiteness, Hopficity}
\begin{abstract}
We prove that if $M$ is a monoid and $A$  a finite set with more than one element, then the residual finiteness of $M$ is equivalent to that of the monoid consisting of all cellular automata over $M$ with alphabet $A$.
\end{abstract}
\date{\today}
\maketitle

\section{Introduction}

In a concrete category,  a \emph{finite object} is an object whose underlying set is finite.   
A \emph{finiteness condition} is a property relative to  the objects of the category that is satisfied by all finite objects.
Finiteness is a trivial example of a finiteness condition. 
Hopficity and co-Hopficity provide  examples of finiteness conditions
that are non-trivial and worth studying in many  concrete categories, e.g.,  the category of groups, the category of rings, the category of compact Hausdorff spaces, etc.
(see the survey paper \cite{varadarajan} and the references therein).   
  We recall that an object $X$ in a concrete category $\CC$ is called
\emph{Hopfian}  if every surjective endomorphism of $X$ is injective and
\emph{co-Hopfian} if every injective endomorphism of $X$ is surjective.
Another interesting finiteness condition
 is \emph{residual finiteness}. An object $X$ in a concrete category $\CC$ is said to be \emph{residually finite} if, given any two distinct elements 
$x_1,x_2 \in X$, there exists a finite object $Y$ of $\CC$ and a $\CC$-morphism 
$\rho \colon X \to Y$ such that
$\rho(x_1) \not= \rho(x_2)$.
\par
Suppose  now that we are given a monoid  $M$  and a finite set $A$.
We say that a map $\tau \colon A^M \to A^M$ is a \emph{cellular automaton} over the monoid $M$ and the \emph{alphabet} $A$ if $\tau$ is continuous for the prodiscrete topology on $A^M$ and $M$-equivariant with respect to the shift action of $M$ on $A^M$ (see Section~\ref{sec:prelim} for more details).
It is clear from this definition that the set $\CA(M,A)$,
consisting of all cellular automata $\tau \colon A^M \to A^M$, 
 is a monoid for the composition of maps.
\par
The main result of the present note is the following statement which yields a characterization of residual finiteness for monoids in terms of cellular automata.

\begin{theorem}
\label{t:CA-res-finite}
Let $M$ be a monoid and let $A$ be a finite set with more than one element.
Then the following conditions are equivalent:
\begin{enumerate}[\rm (a)]
\item
the monoid $M$ is residually finite;
\item
the monoid $\CA(M,A)$ is residually finite.
\end{enumerate}
\end{theorem}

Residual finiteness is obviously hereditary, in the sense that every subobject of a residually finite object is itself residually finite.
Thus, an immediate consequence of implication (a) $\Rightarrow$ (b) in 
Theorem~\ref{t:CA-res-finite} is the following: 

\begin{corollary}
\label{c:subsemigroup-res-fin}
Let $M$ be a residually finite monoid and let $A$ be a finite set.
Then every subsemigroup of  $\CA(M,A)$ is residually finite.
\qed
\end{corollary}

In \cite{malcev-representation}, it was shown by Mal'cev that 
every  finitely generated residually finite semigroup  is Hopfian and  
has a residually finite monoid of endomorphisms. 
Combining Corollary~\ref{c:subsemigroup-res-fin} with these results of Mal'cev,
we get  the following.

\begin{corollary}
\label{c:sub-hopfian}
Let $M$ be a residually finite monoid and let $A$ be a finite set.
Then every finitely generated subsemigroup of $\CA(M,A)$ is Hopfian.
\qed
\end{corollary}

\begin{corollary}
\label{c:endo-subsemi-res-fin}
Let $M$ be a residually finite monoid and let $A$ be a finite set.
Suppose that $T$ is a  finitely generated subsemigroup of $\CA(M,A)$.
Then the monoid $\End(T)$ of endomorphisms of $T$ is residually finite. 
 \qed
\end{corollary}

The next section precises the terminology used and collects some background material.
For the convenience of the reader, we have also included a proof of the results of Mal'cev mentioned above.
The proof of Theorem~\ref{t:CA-res-finite} is given in the final section.

\section{Preliminaries}
\label{sec:prelim}

\subsection{Semigroups and monoids}
A \emph{semigroup} is a set equipped with an associative binary operation. 
We shall use a multiplicative notation for the operation on semigroups.
If $S$ and $T$ are semigroups, a \emph{semigroup morphism} from $S$ to $T$ is a map 
$\varphi \colon S \to T$ such that
$\varphi(s_1 s_2) = \varphi(s_1) \varphi(s_2)$ for all $s_1,s_2 \in S$.
We denote by $\Mor(S,T)$ the set consisting of all semigroup morphisms from $S$ to $T$.
A relation $\gamma$ on a semigroup $S$ is called a \emph{congruence relation} if there exist a semigroup $T$ and a semigroup morphism $\varphi \colon S \to T$ such that 
$\gamma$ is the \emph{kernel relation} associated with  $\varphi$, i.e., the equivalence relation defined by
$$
\gamma := \{(s_1,s_2) \in S \times S : \varphi(s_1) = \varphi(s_2) \}.
$$
Equivalently, an equivalence relation $\gamma \subset S \times S$ on $S$ is a congruence relation if and only if $(s_1,s_2) \in \gamma$ implies $(s s_1,s s_2) \in \gamma$ and $(s_1 s, s_2 s) \in \gamma$ for all $s,s_1,s_2 \in S$.
\par
Suppose that $\gamma$ is a congruence relation on a semigroup $S$. 
 Then there is a natural semigroup structure on the quotien set $S/\gamma$.
This semigroup structure is the only one  for which the canonical map from $S$ onto 
$S/\gamma$ (i.e., the map sending each $s \in S$ to its $\gamma$-class $[s] \in S/\gamma$) is a semigroup morphism.
Moreover, $\gamma$ is the kernel relation associated with this semigroup morphism.
One says that the congruence relation $\gamma$ is of \emph{finite index} if the quotient semigroup $S/\gamma$ is finite.
\par
A \emph{monoid} is a semigroup admitting an identity element.
The identity element of a monoid $M$ is denoted $1_M$.
If $M$ and $N$ are monoids, a monoid morphism from $M$ to $N$ is a semigroup morphism from $M$ to $N$ that sends $1_M$ to $1_N$.
Suppose that  $\gamma$ is a congruence relation on a monoid $M$. 
Then the quotient semigroup 
$M/\gamma$ is a monoid. Moreover, the canonical semigroup morphism  from $M$ onto $M/\gamma$ is a monoid morphism. 
\par

\subsection{Residually finite semigroups}
It is clear from the general definition of residual finiteness given in the Introduction 
that a group is residually finite as a group if and only if it is residually finite as a monoid and that a monoid is residually finite as a monoid if and only if it is residually finite as a semigroup. 
\par
The class of residually finite semigroups includes all free groups and hence (since residual finiteness is a  hereditary property) all free monoids and all free semigroups,
all polycyclic groups~\cite{hirsch} and hence all finitely generated nilpotent groups, 
all finitely generated commutative semigroups~\cite{malcev} (see also~\cite{lallement} and~\cite{carlisle}),
all finitely generated  semigroups that are both regular in the sense of von Neumann and nilpotent in the sense of Mal'cev~\cite{lallement-nilpotent},
and all finitely generated semigroups of matrices over commutative 
rings~\cite{malcev-representation}, \cite{stallings-notes}.    
\par
The following two fundamental results about finitely generated residually finite semigroups are due to Mal'cev~\cite{malcev-representation} (see also~\cite{evans-1970}).
 
\begin{theorem}[Mal'cev]
\label{t:res-fin-hopfien}
Every finitely generated residually finite semigroup is Hopfian.
\end{theorem}

\begin{proof}
Let $S$ be a finitely generated residually finite semigroup.
Suppose that $\psi \colon S \to S$ is a surjective endomorphism of $S$.
Let $s_1$ and $ s_2$ be distinct elements in $S$.
Since $S$ is residually finite, there exists a finite semigroup $T$ and a semigroup morphism
$\rho \colon S \to T$ such that $\rho(s_1) \not= \rho(s_2)$. 
   Consider the map
$$
\Phi \colon \Mor(S,T) \to \Mor(S,T)
$$
 defined by $\Phi(u) = u \circ \psi$ for all $u \in \Mor(S,T)$.
Observe that  $\Phi$ is injective since $\psi$ is surjective. 
On the other hand, 
as $S$ is finitely generated and $T$ is finite, the set $\Mor(S,T)$ is finite. 
Therefore  $\Phi$ is also surjective. In particular, there exists a morphism
$u_0 \in \Mor(S,T)$ such that $\rho = \Phi(u_0) = u_0 \circ \psi$. 
Since $\rho(s_1)  \not= \rho(s_2)$,
this implies that $\psi(s_1) \not= \psi(s_2)$.
We deduce that $\psi$ is injective.
This  shows that $S$ is Hopfian.
\end{proof}

\begin{theorem}[Mal'cev]
\label{t:aut-res-fini}
Let $S$ be a finitely generated residually finite semigroup.
Then the monoid $\End(S)$ is residually finite.
\end{theorem}

Let us first establish the following auxiliary result.

\begin{lemma} 
\label{l:int-ss-gpes-ind-fini}
Let $S$ be a semigroup. 
Suppose that  $\gamma_1$ and $\gamma_2$ are congruence relations of finite index 
on $S$. 
Then the congruence relation $\gamma := \gamma_1 \cap \gamma_2$ is also of finite index   on $S$.
\end{lemma}

\begin{proof}
Two elements in $S$ are congruent modulo $\gamma$  if and only if 
they are both congruent  modulo $\gamma_1$ and modulo $\gamma_2$. Therefore, there is an injective map from 
$S/\gamma$ into $S/\gamma_1 \times S/\gamma_2$ given 
by $[s] \mapsto ([s]_1,[s]_2)$,
where $[s]$ (resp.~$[s]_1$, resp.~$[s]_2$) denote the class of $s \in S$ modulo
$\gamma$ (resp.~$\gamma_1$, resp.~$\gamma_2$). 
As the sets $S/\gamma_1$ and $S/\gamma_2$ are finite by our hypothesis, we deduce that
$S/\gamma$ is also finite, that is, $\gamma$ is of finite index on $S$.
\end{proof}

 \begin{proof}[Proof of Theorem \ref{t:aut-res-fini}]
Let $\alpha_1, \alpha_2 \in \End(S)$ such that $\alpha_1 \not= \alpha_2$. 
Then we can find an element $s_0 \in S$ such that $\alpha_1(s_0) \not= \alpha_2(s_0)$. 
As $S$ is residually finite, there exist
a finite semigroup $T$ and a semigroup morphism $\rho \colon S \to T$ satisfying
$\rho(\alpha_1(s_0))  \not= \rho(\alpha_2(s_0))$. Consider the set 
$\gamma \subset S \times S$ defined by
$$
\gamma := \bigcap_{\psi \in \Mor(S,T)} \gamma_\psi,
$$
where $\gamma_\psi$ denotes the kernel congruence relation associated with the 
semigroup morphism $\psi \colon S \to T$.
Observe first that $\gamma$ is a congruence relation  on $S$ since it is the intersection of a family of congruence relations on $S$.
On the other hand, for every $\alpha \in \End(S)$ and $(s_1,s_2) \in \gamma$, 
we have that
$(\alpha(s_1),\alpha(s_2)) \in \gamma$ since $\psi \circ \alpha \in \Mor(S,T)$ for every 
$\psi \in \Mor(S,T)$.
We deduce that $\alpha$ induces an endomorphism 
$\overline{\alpha}$ of $S/\gamma$, given by 
$\overline{\alpha}([s]) = [\alpha(s)] $, for all $s \in S$ (here $[s]$ denotes the $\gamma$-class of $s$).
The map $\alpha \mapsto \overline{\alpha}$ is clearly a morphism from $\End(S)$ into
$\End(S/\gamma)$. 
Now the set $\Mor(S,T)$ is finite since $S$ is finitely generated and $T$ is finite.
Moreover, as the semigroup $T$ is finite, the congruence relation $\gamma_\psi$ is of finite index on $S$ for every $\psi \in \Mor(S,T)$.
By applying  Lemma~\ref{l:int-ss-gpes-ind-fini},
  we deduce that the congruence relation $\gamma$ is of finite index on $S$. 
Thus, the semigroup $S/\gamma$ is finite
and hence the monoid $\End(S/\gamma)$ is also finite. On the other hand, we have
that 
$$
\overline{\alpha_1}([s_0] ) = [\alpha_1(s_0)]  \not= [\alpha_2(s_0)] = \overline{\alpha_2}([s_0]) 
$$
since $\gamma \subset \gamma_\rho$  and $\rho(\alpha_1(s_0)) \not= \rho(\alpha_2(s_0))$. 
Therefore
$\overline{\alpha_1} \not= \overline{\alpha_2}$. 
This shows that the monoid $\End(S)$ is residually finite.
\end{proof}

\subsection{Shift spaces}
Let $A$ be a finite set, called the \emph{alphabet}, and let $M$ be a monoid.
The set $A^M$, consisting of all maps $x  \colon M \to A$,
is called the set of \emph{configurations} over the monoid $M$ and the alphabet $A$.
We equip $A^M$ with its \emph{prodiscrete topology}, i.e., the product topology obtained by taking the discrete topology on each factor $A$ of
$A^M = \prod_{m \in M} A$.
Observe that  $A^M$ is a compact Hausdorff totally disconnected space since it is a product of compact Hausdorff totally disconnected spaces.
We also equip $A^M$ with the $M$-\emph{shift}, that is,
the action of the monoid $M$ on $A^M$ given by $(m,x) \mapsto m x$, where
$$
m x (m') = x(m' m)
$$
for all $x \in A^M$ and $m,m' \in M$.
\par
Let $\gamma$ be a congruence relation on $M$. 
We define the subset $\Inv(\gamma) \subset A^M$ by
$$
\Inv(\gamma) := \{ x \in A^M : m_1 x = m_2 x \text{  for all } (m_1,m_2) \in \gamma \}.
$$
Observe that $\Inv(\gamma)$ is $M$-\emph{invariant}, i.e., $m x \in \Inv(\gamma)$ for all $m \in M$ and $x \in \Inv(\gamma)$.
One immediately checks that $\Inv(\gamma)$ consists of all configurations $x \in A^M$ that are constant on each $\gamma$-class. This implies in particular that 
the set $\Inv(\gamma)$ is finite whenever $\gamma$ is of finite index. 
\par
A configuration $x \in A^M$ is called \emph{periodic} if its orbit
$$
M x := \{ m x : m \in M \}
$$
is finite.

Residually finite monoids are characterised by the density of periodic configurations in their shift spaces. More precisely, we have the following result
(see~\cite[Proposition 2.14]{ccs-surjunctive-monoids}).

\begin{theorem}
\label{t:res-finite-dynam}
Let $M$ be a monoid and let $A$ be a finite set
with more than one element.
Then the  following conditions are equivalent:
\begin{enumerate}[{\rm (a)}]
\item 
the monoid $M$ is residually finite;
\item 
the  set   of  periodic configurations of $A^M$   is
dense in $A^M$ for the prodiscrete topology.
\end{enumerate}
\qed
\end{theorem}

\subsection{Cellular automata}
Let $M$ be a monoid and let $A$ be a finite set.
A \emph{cellular automaton} over the monoid $M$ and the alphabet $A$ is a map 
$\tau \colon A^M \to A^M$
that is continuous for the prodiscrete topology on $A^M$ and
commutes with the shift action, i.e., satisfies $\tau( m x ) = m \tau(x) $ for all $m \in M$ and $x \in A^M$.
We denote by $\CA(M,A)$ te set consisting of all cellular automata $\tau \colon A^M \to A^M$.
It is clear from the above definition that $\CA(M,A)$ is a monoid for the composition of maps.

\begin{example}
\label{ex:tau-m}
If $ m \in M$, one immediately checks that  the map $\tau_m \colon A^M \to A^M$, defined by 
$\tau(x) = x \circ L_m$ for all $x \in A^M$, where $L_m \colon M \to M$ denotes the left-multiplication by $m$, is a cellular automaton.
Moreover, the map 
$m \to \tau_m$ yields an anti-monoid morphism from $M$ into $\CA(M,A)$.
This means that
$\tau_{1_M}$ is the identity map on $A^M$ and that
$\tau_{m_1 m_2} = \tau_{m_2} \circ \tau_{m_1}$ for all $m_1,m_2 \in M$.
This monoid anti-morphism is injective as soon as the alphabet $A$ has more than one element.
Indeed, let $m_1 ,m_2 \in M$ with $m_1 \not= m_2$.
Suppose that
  $a$ and $b$ are distinct elements in $A$ and consider the configuration $x \in A^M$ defined by $x(m_1) = a$ and $x(m) = b$ for all $m \in M \setminus \{m_1\}$.
  We then have $\tau_{m_1}(x) \not= \tau_{m_2}(x)$ since
  $$
  \tau_{m_1}(x)(1_M) = x(m_1) = a \not= b = x(m_2) = \tau_{m_2}(x)(1_M),
  $$
  and hence $\tau_{M_1} \not= \tau_{m_2}$.
  \end{example}

\section{Proof of the main  result}

In this section, we give the proof of Theorem~\ref{t:CA-res-finite}.

\begin{proof}[Proof of (a) $\Rightarrow$ (b)]
Suppose that $M$ is residually finite.
Let $\tau_1, \tau_2 \in \CA(M,A)$ be two distinct cellular automata. 
\par
Since $M$ is residually finite, the periodic configurations in $A^M$ are dense in $A^M$
(see Theorem~\ref{t:res-finite-dynam}).
 As $\tau_1$ and $\tau_2$ are continuous and $A^M$ is Hausdorff,
this implies that there exists a periodic configuration $x_0 \in A^M$ such that 
$\tau_1(x_0) \neq \tau_2(x_0)$.
Consider the orbit $Y := Mx_0$ of $x_0$ under the $M$-shift.
As the set $Y$ is $M$-invariant, the equivalence relation $\gamma$ defined by
$$
\gamma := \{(m_1,m_2) \in M \times M : m_1 y = m_2 y \quad \text{for all }  y \in Y \} \subset M \times M
$$
is a congruence relation on $M$. 
Moreover, $\gamma$  is of finite index since $Y$ is finite.
Consider now  the associated $M$-invariant subset
$$
X := 
\Inv(\gamma) = \{x \in A^M : m_1 x = m_2 x \quad 
\text{ for all } (m_1,m_2) \in \gamma\} \subset A^M.
$$
Note that $X$ is finite since the congruence relation $\gamma$ is of finite index.
As every cellular automaton $\tau \in \CA(M,A)$ is $M$-equivariant,
restriction to $X$ yields a monoid morphism 
$\rho \colon \CA(M,A) \to \Map(X)$, where $\Map(X)$ denotes the \emph{symmetric monoid} of $X$, i.e., the set consisting of all maps $f \colon X \to X$ with the composition of maps as the monoid operation.
Observe that the monoid $\Map(X)$ is finite since $X$ is finite.
On the other hand, as $x_0 \in Y \subset X$ and $\tau_1(x_0) \not= \tau_2(x_0)$,
we have that $\rho(\tau_1) \not= \rho(\tau_2)$.
This shows that $\CA(M,A)$ is residually finite.  
\end{proof}

\begin{proof}[Proof of (b) $\Rightarrow$ (a)]
First observe that a semigroup is residually finite if and only if its opposite semigroup is
(this trivially follows from the fact that a semigroup is finite if and only if its opposite semigroup is).
Suppose now that the monoid $\CA(M,A)$ is residually finite.
Since there is an injective monoid anti-morphism $M \to \CA(M,A)$ 
(see Example~\ref{ex:tau-m})
and residual finiteness is hereditary, 
we deduce that the opposite monoid of $M$ is residually finite.
By the above observation, the monoid $M$ is itself  residually finite. 
\end{proof}

\begin{remark}
Let us observe that Corollary~\ref{c:sub-hopfian} and 
Corollary~\ref{c:endo-subsemi-res-fin}
become false if we drop the hypothesis that the subsemigroup of $\CA(M,A)$ is finitely generated, even if we restrict to the case where $M$ is the group $ \Z$ of integers
(the classical case studied in symbolic dynamics).
Indeed, let $A$ be a finite set with more than one element.
It can be shown, using the technique of \emph{markers} introduced in~\cite{hedlund}, that
the free group on two generators can be embedded in $\CA(\Z,A)$
(see~\cite[Theorem~2.4]{boyle-lind-rudolph} for a more general statement).
It follows that the free group $F_\infty$ on infinitely many generators $g_i$, $i \in \N$, 
can be also embedded in $\CA(\Z,A)$.
Now, the group $F_\infty$ is not Hopfian since the unique endomorphism 
$\psi \in \End(F_\infty)$   
satisfying $\psi(g_i) = g_{i - 1}$ if $i \geq 1$ and $\psi(g_0) = g_0$ 
is clearly surjective but not injective.
On the other hand, by using automorphisms of $F_\infty$ induced by permutations of its generators,
one sees that the automorphism group of $F_\infty$ contains a copy of the symmetric group $\Sym(\N)$ (the group of permutations of $\N$).
The  group  $\Sym(\N)$ is not residually finite
since, by Cayley's theorem, every countable group can be embedded in $\Sym(\N)$ and there exist countable groups that are not residually finite
(e.g., the additive group $\Q$ of rational numbers or the Baumslag-Solitar group
$BS(2,3) := \langle a,b : b a^2 b^{-1} = a^3 \rangle$).
Therefore, the monoid  $\End(F_\infty)$ is not residually finite either. 
\end{remark}

\end{document}